\documentclass[11pt]{article}
\usepackage{amssymb,amsmath,amsthm,amsfonts}
\usepackage{stmaryrd}
\usepackage[english]{babel}
\usepackage{color}
\usepackage{graphicx}
\usepackage{dsfont}
\usepackage[top=1in, bottom=1.25in, left=1.25in, right=1.25in]{geometry}
\usepackage[colorlinks=true]{hyperref}
 
\newtheorem{theorem}{Theorem} 
\newtheorem{remark}{Remark}

\newtheorem{lemma}{Lemma}

\newtheorem{proposition}{Proposition}

\makeatletter % `@' now normal "letter"
\@addtoreset{equation}{section}
\makeatother  % `@' is restored as "non-letter"

\newcommand{\R}{\mathbb{R}}
\newcommand{\Z}{\mathbb{Z}}
\newcommand{\df}{\displaystyle\frac}
\newcommand{\Jac}{\mathbf{Jac}}
\newcommand{\diag}{\mathbf{diag}}

\DeclareMathOperator{\Ima}{Im}

\begin{document}

\title{Sharp seasonal threshold property for cooperative population dynamics with concave nonlinearities}

\author{Hongjun Ji \and  Martin Strugarek}

%\address{Universit\'{e} P.-M. Curie, Paris 6, Laboratoire Jacques--Louis Lions,
%UMR 7598 CNRS, 4 Pl. Jussieu, 16-26, 75252 Paris, France}
%\email{hongjun.ji@upmc.fr, martin.strugarek@upmc.fr}

\date{}

\maketitle

%\tableofcontents

\begin{abstract}
We consider a biological population whose environment varies periodically in time, exhibiting two very different ``seasons'': one is favorable and the other one is unfavorable. For monotone differential models with concave nonlinearities, we address the following question: the system's period being fixed, under what conditions does there exist a critical duration for the unfavorable season? By ``critical duration'' we mean that above some threshold, the population cannot sustain and extincts, while below this threshold, the system converges to a unique periodic and positive solution. We term this a ``sharp seasonal threshold property" (SSTP, for short).

Building upon a previous result, we obtain sufficient conditions for SSTP in any dimension and apply our criterion to a two-dimensional model featuring juvenile and adult populations of insects.

\end{abstract}



\section{Introduction }
We study differential dynamical systems arising from nonlinear periodic positive differential equations of the form
\begin{equation} \label{F0}
\frac{dx}{dt}=F(t,x),
\end{equation}
where $F$ is monotone and concave. These systems exhibit well-known contraction properties when $F$ is continuous (see \cite{HW}, \cite{HLS1}, \cite{JJF}). We extend in Theorem \ref{thm:extension} these properties to non-linearities that are only piecewise-continuous in time. This extension is motivated by the study of typical seasonal systems in population dynamics.

We denote by $\theta \in [0, 1]$ the proportion of the year spent in unfavorable season. Then, we convene that time $t$ belongs to an unfavorable (resp. a favorable) season if $nT   \leq t < (n + \theta)T$ (resp. if $ (n+ \theta)T \leq  t <  (n+1)T$) for some $n \in \Z_+$. In other words, we study the solutions to:
\begin{equation}\label{dyn:2seasons}
\frac{dX}{dt} = G(\pi_{\theta}(t), X), \quad \pi_{\theta}(t) = \begin{cases}
                                              \pi^U \text{ if } \frac{t}{T} - \lfloor \frac{t}{T} \rfloor \in [0, \theta),
                                              \\
                                              \pi^F \text{ if } \frac{t}{T} - \lfloor \frac{t}{T} \rfloor \in [\theta, 1),
                                             \end{cases}
\end{equation}
for some $G : \mathcal{P} \times \R^N \to \R^N$, with $\pi^U, \pi^F \in \mathcal{P}$ where $\mathcal{P}$ is the parameter space. We are looking for conditions ensuring that a sharp seasonal threshold property holds, that is:
\begin{equation*}
  \exists \theta_* \in [0, 1] \text{ such that } \begin{cases}
                                                  \text{if } \theta < \theta_*, \exists ! q : \R_+ \to \R^N, T\text{-periodic}, q \gg 0 \text{ and } \\ \forall X_0 \in \R_+^N \backslash \{0 \}, X \text{ converges to } q,
                                                  \\
                                                  \text{if } \theta > \theta_*, \forall X_0 \in \R_+^N, \, X \text{ converges to } 0.
                                                 \end{cases}
 \tag{SSTP}
 \label{def:threshold}
\end{equation*}
Ecologically, the respective duration of dry and wet seasons is crucial for population sustainability in various species. The property \eqref{def:threshold} means that if the dry season is longer than $\theta_* T$ then the population collapses and if it is shorter then the population densities will tend to be periodic.

Assume  that $F(t, 0) \equiv 0$. Thanks to the contraction properties of concave nonlinearities, the whole problem reduces to the study of the Floquet eigenvalue with maximum modulus of the linearization of \eqref{F0} at $X = 0$:
 \begin{equation} \label{dyn:linearized}
\frac{dz}{dt}=D_xF(t,0)z.
\end{equation}
In fact, this eigenvalue is equal to the spectral radius of the Poincar\'{e} application for~\eqref{dyn:linearized}, which we compute here for piecewise-autonomous systems.

Our proof uses the Perron-Frobenius theorem and relies on the Perron eigenvalue and (left and right) eigenvectors. The importance of this eigenvalue for quantifying the effects of seasonality has been acknowledged continuously in mathematical biology in at least three application fields: circadian rhythms (in particular in connection with cell division and tumor growth), harvesting and epidemiology.

It was noted in \cite{CGP} that Floquet eigenvalue with maximum modulus of \eqref{dyn:linearized} is always larger that the Perron eigenvalue of some averaged (over a period) matrix $\overline{F}$ defined from the entries of $D_x F(t, 0)$. There has been a continued interest in this eigenvalue for linear models of cell division since and we refer to \cite{GL} in particular for a detailed study of the monotonicity of the Perron eigenvalue with respect to parameters of a structured model for cell division. In a stochastic framework for growth and fragmentation, \cite{CCF} establishes a similar monotonicity property. In this context, the Perron eigenvalue is seen as the cell growth rate, and this is why its dependence in the model parameters is important. Here, we connect the eigenvalue monotonicity with a non-extinction condition to derive the~\eqref{def:threshold}. We emphasize that our Theorem \ref{thm:anydimension} gives some sufficient conditions for the monotonicity of the Perron eigenvalue, in the case when there are only two different seasons.

In dimension $1$, for the logistic equation with harvesting, Xiao has shown in \cite{XD} a sharp threshold property, where the two different ``seasons" correspond to one harvesting period (''unfavorable season``) and one rest period (''favorable season``). Contrary to the case of cell division, the model treated there is non-linear, though $1$-dimensional. Our results extend a part of those of \cite{XD} to $n$-dimensional concave monotone systems. Note that the cited article also studies the maximal sustainable yield, which can be seen as an objective function of the periodic solution $q$. On this topic, \cite[Section 5]{MPS} studies a structured problem of adaptive dynamics with concave nonlinearity and periodic forcing to show a similar effect as in \cite{XD} (there, for population size): in both cases, time fluctuations can improve an objective value.

For applications in epidemiology, where seasonality often has dramatic effects, we refer to \cite{BAD} and \cite{B} for the computation of case reproduction numbers with seasonal forcing.

The organization of the paper is as follows. The motivating model is detailed in Section~\ref{sec:motivation}, where we also define some notations. In Section \ref{sec:mainres} we state our results: first (Theorem \ref{thm:extension}) an extension to piecewise-continuous nonlinearities of the well-known results on monotone concave nonlinearities, then (Theorem \ref{thm:anydimension}) fairly general sufficient conditions for systems in any space dimension $N \in \Z_{>0}$ to satisfy ~\eqref{def:threshold}, and finally (Theorem \ref{thm:main}) an application to the two-dimensional system \eqref{dyn:2seasons}, for which we are able to show the threshold property \eqref{def:threshold} for a wide range of parameters.
The proofs are detailed in Section~\ref{sec:proofs} (and in Appendix \ref{app:proof} for Theorem~\ref{thm:extension}), while extensions and possible research directions are gathered in Section \ref{sec:discussion}.

\section{Context and motivation}
\label{sec:motivation}

Our reference model is a simplistic description of the population dynamics of some insects, with a juvenile stage exposed to quadratic competition and an adult stage. Let $J(t), A(t)$ represent the populations of juveniles and adults at time $t$, respectively. A very simple dynamic is defined by
\begin{equation}\label{S0}
\left\{
\begin{array}{l}
\df{d J}{dt}=b A - J (h + d_J + c_J J),
\\[10pt]
\df{d A}{dt}= h J- d_A A,
\end{array}
\right.
\end{equation}
where $d_Y$ ($Y \in \{J, A \}$) stands for the (linear) death rate, $b$ is the birth rate, $h$ is the hatching rate and the parameter $c_J$ tunes the only non-linearity: quadratic competition (=density-dependent death rate) among juveniles. This term effectively limits the total population size, as we will prove below. We use it to represent resource limitation both for breeding sites availability and for nutrient availability during growth. In principle, the parameters may depend on time:
\begin{equation}
  \forall t \in \R, \quad \pi(t) := (b, h, d_J, c_J, d_A) \in \R_+^5.
  \label{def:pi}
\end{equation}
For convenience, we rewrite the right-hand side of \eqref{S0} as $G(\pi, X)$ with $X = (J, A) \in \R^2$, and $G : \R_+^5 \times \R^2 \to \R^2$.

In the tempered areas where mosquito populations are established, dramatic seasonal variations in population abundance are usually observed. Namely, there is explosive growth in summer after rain events, whereas mosquitoes are very scarce in winter. This phenomenon is possible thanks to dormant (or "quiescent" or "refuge") phases in the mosquito's life-cycle.
These seasonal variations imply that the natural environment (temperature, rainfall, humidity etc.)  is very important for the mosquito.

We propose to study population dynamics in simple models such as \eqref{S0} under periodic seasonal forcing. As a rough approximation, we set up \eqref{S0} with periodic piecewise-constant coefficients of period $T = 1 \textrm{ year}$, each one possibly taking two different values over one period. Thus, the year is divided into unfavorable and favorable seasons, defined by parameter values $\pi^U, \pi^F \in \R_+^5$ such that 
\begin{equation}
\begin{pmatrix}
-d_J^F + d_J^U & b^F - d_A^F - (b^U -d_A^U) \\[10pt]
h^F - h^U & -d_A^F + d_A^U
\end{pmatrix} > 0.
\label{hyp:parameters}
\end{equation}
The four scalar inequalities of condition \eqref{hyp:parameters} deserve a biological justification. It implies that during the favorable season, the hatching rate is larger than during the unfavorable season, while death rates (for juveniles, and adults) are smaller. These assumptions rely on the facts that breeding sites availability and quality is much higher in good season (whence higher hatching rate and birth rate and lower juvenile competition), while the temperature increase can be expected to extend the life-span of both adults and juveniles. The first component in \eqref{hyp:parameters} implies that the growth coefficients $b - d_A$ are ordered: $b^F - d_A^F > b^U - d_A^U$. This is true in particular if $b^F > b^U$, but holds in more generality.

We emphasize that the systems under study are excessively simple because, in mathematical terms, they are cooperative with concave nonlinearity, and as such they have strong asymptotic convergence properties.

\paragraph{Notations.} Let $X$ be a normed vector space and $\mathcal{K} \subset X$ be a cone.
For  $A , B \in X$, we define
\begin{align*}
& A \geq_{\mathcal{K}} B \Longleftrightarrow B-A \in \mathcal{K},
\\[5pt]
&A >_{\mathcal{K}} B \Longleftrightarrow A \geq_{\mathcal{K}} B \text{ and } A \not= B,
\\[5pt]
 &A\gg_{\mathcal{K}} B   \Longleftrightarrow B - A \in \mathring{\mathcal{K}}.
\end{align*}
In the special case when $X = \R^{m \times n}$ and $\mathcal{K} = \R_+^{m \times n}$ for some $m, n \in \Z_{>0}$, we omit the subscript $\mathcal{K}$. For $A, B \in X$ such that $A \leq_{\mathcal{K}} B$, the interval $[A,B]$ is a non-empty set defined by
\[
    [A, B] = \{ C \in X, \, A \leq_{\mathcal{K}} C \leq_{\mathcal{K}} B \}.
\]

For $A \in \mathcal{M}_n(\R)$, the spectral radius of  $A$ denoted by  $\rho(A)$ is $\rho(A):= \max \{ \lvert \lambda  \rvert, \lambda \in \sigma(A)\}$ where  $\sigma(A)$  is the spectrum of $A$ and the spectral abscissa of $A$, denoted by $\mu(A)$, is 
$\mu(A) := \max \{ \Re(\lambda), \,\lambda \in \sigma(A) \}$. 

A matrix is called irreducible if it is not similar to an upper triangular matrix by permutation. We call Metzler the matrices in $\mathcal{M}_n(\R)$ whose off-diagonal elements are all nonnegative.

For $X, Y$ two real finite-dimensional vector spaces embedded in $\R^d$ ($d \geq 1$), we denote by $\mathcal{L} (X, Y)$ the space of linear applications from $X$ to $Y$, with the convention $\mathcal{L} (X) = \mathcal{L} (X, X)$.
We denote the adjoint of $A \in \mathcal{L}(X,Y)$ by $A^* \in \mathcal{L}(Y,X)$, defined by
\[
    \forall (v, w) \in X \times Y, \quad \langle Av, w \rangle = \langle v, A^* w \rangle,
\]
where $\langle \cdot, \cdot \rangle$ denoted the euclidean scalar product in $\R^d$.
For $x \in \R$, the notation $\lfloor x \rfloor$ stands for the largest integer $n \in \Z$ such that $n \leq x$.

Let $F : \R_t \times \R^N_x \to \R^N$ be piecewise continuous in $t$ and continuously differentiable in~$x$. The system
\eqref{F0} is {\itshape cooperative } if its Jacobian matrix is  Metzler:
\[
\forall(t,x)  \in \mathbb{R}_+ \times \mathbb{R}_+^N,  \,  i \ne j  \implies  \frac{ \partial F_i}{\partial x_j} (t, x) \geq 0,
\tag{M}
\label{monotone}
\]
It is {\itshape positive}  ({\itshape i.e.}, $\R_+^N$ is an invariant set) if
\[
\forall t \in \R_+, \, \forall 1 \leq i \leq N, \, \forall x \geq 0, \quad x_i = 0 \implies F_i(t, x) \geq 0.
\tag{P}
\label{positive}
\]
Under condition \eqref{monotone}, \eqref{F0} is positive if $\forall t \in \R_+, \, F(t,0)\geq 0$.
We say that \eqref{F0} defines a {\itshape concave dynamics} on $\R_+^N$ if
\[
\forall 0 \ll x \ll y, \,    D_xF(t,x)  \geq  D_xF(t,y),
\tag{C}
\label{concave}
\]
and that \eqref{dyn:linearized} is {\itshape irreducible} if  
\begin{equation*}
\forall t \in \mathbb{R}_+, \, D_xF(t,0) \text{ is  irreducible in } \mathcal{M}_N (\R).
\tag{I}
\label{irre}
\end{equation*}

\section{Results}
\label{sec:mainres}

\subsection{General results}
In order to study the asymptotic behavior of \eqref{dyn:2seasons}, we generalize a result by Smith~\cite{HLS1} (refined by Jiang in \cite{JJF}) about continuous concave and cooperative nonlinearities to piecewise-continuous (in time) nonlinearities.
\begin{theorem}  
Let $F : \R_t \times \R^N_x \to \R^N$ be $T$-periodic and piecewise-continuous in $t$ and such that for all $t \in \R_+$, $F(t, \cdot) \in \mathcal{C}^1 (\R^N, \R^N)$. Assume that $F$ satisfies assumptions \eqref{positive}, \eqref{monotone}, \eqref{concave} and \eqref{irre}, so that the associated differential system \eqref{F0} is positive, monotone and concave with irreducible linearization at $0$. Let $\lambda \in \R$ denote the Floquet multiplier with maximal modulus of \eqref{dyn:linearized}.

%We consider the $T$-periodic differential system defined by $F$:
%\begin{equation}\label{dyn:Tperpiecewise}
%\frac{dx}{dt}=F(t,x),
%\end{equation}
%and assume that \eqref{dyn:Tperpiecewise} defines a cooperative and concave dynamics on $\R_+^N$, that is:
%\begin{align*}
%&\forall t \in \R_+, \, \forall 1 \leq i \leq N, \, \forall x \geq 0, \, x_i = 0 \implies F_i(t, x) \geq 0,
%\tag{P}
%\label{positive}
%\\[10pt]
%&\forall(t,x)  \in \mathbb{R}_+ \times \mathbb{R}_+^N,  \,  i \ne j  \implies  \frac{ \partial F_i}{\partial x_j} (t, x) \geq 0,
%\tag{M}
%\label{monotone}
%\\[10pt]
%&\forall 0 \ll x \ll y, \,    D_xF(t,x)  \geq  D_xF(t,y).
%\tag{C}
%\label{concave}
%\end{align*}
%Then, every solution of \eqref{dyn:Tperpiecewise} with $x(t_0) \geq 0$ extends to $[t_0, \infty)$ with $x(t) \geq 0$ for all $t\geq t_0$.
%
%Assume in addition that $F(t, 0) \equiv 0$. The variational equation about $0$ is the following $T$-periodic linear equation:
% \begin{equation} \label{dyn:linearized}
%\frac{dz}{dt}=D_xF(t,0)z.
%\end{equation}
%Assume that \eqref{dyn:linearized} is irreducible:
%\begin{equation*}
%\forall t \in \mathbb{R}_+, \, D_xF(t,0) \text{ is  irreducible},
%\tag{I}
%\label{irre}
%\end{equation*}
%and let $\lambda \in \R$ denote the Floquet multiplier with maximal modulus of \eqref{dyn:linearized}.

If $\lambda \leq 1$ then every non-negative solution of \eqref{F0} converges to $0$. Otherwise,
\begin{enumerate}
 \item[(i)] either every non-negative solution of \eqref{F0} satisfies  $\lim\limits_{t \to \infty }x(t)=  \infty $,
 \item[(ii)] or \eqref{F0} possesses a unique (nonzero) $T$-periodic solution $q(t)$.
\end{enumerate}
  In case $(ii)$, $q \gg 0$ and $\lim\limits_{t \to \infty }(x(t)-q(t))= 0 $ for every non-negative solution of~\eqref{F0}.
\label{thm:extension}
\end{theorem}
The proof of Theorem \ref{thm:extension} (in Appendix \ref{app:proof}) follows closely the lines of \cite{HLS1} and \cite{JJF}.

An illuminating example when Theorem \ref{thm:extension} applies is for $T$-periodic piecewise autonomous differential systems, where for all $x \in \R^N$, $F(\cdot, x)$ is a piecewise-constant function. Namely, we assume that there exists $K \in \Z_{>0}$ and functions $(F^k)_{1 \leq k \leq K} : \R_+^N \to \R_+^N$ such that:
 \begin{equation}
F(t,x) = F^{k} (x) \text{ if } \frac{t}{T} - \big\lfloor \frac{t}{T} \big\rfloor \in [\theta_{k-1}, \theta_k),
\label{def:Fpiecewiseconstant}
 \end{equation}
where $(\theta_k)_{0 \leq k \leq K} \in [0, 1]^{K+1}$ is a non-decreasing family such that $\theta_0 = 0$ and $\theta_K = 1$. 
To verify the hypotheses of Theorem \ref{thm:extension}, we need to assume that for all $1 \leq k \leq K$,  $F^k$ is continuously differentiable, monotone, concave and satisfies $F^k (0) = 0$; and in addition that $D F^k (0)$ is irreducible for all $1 \leq k \leq K$.

The main advantage of piecewise-constant non-linearities is that for such dynamics (and almost only for these dynamics), the Floquet multiplier with maximal modulus $\lambda$ can be computed explicitly as the following spectral radius:
\begin{equation}
 \lambda = \rho \big( e^{(\theta_K-\theta_{K-1}) T \cdot D F^K(0)} \cdots e^{(\theta_1 - \theta_0) T \cdot D F^1 (0)} \big).
 \label{formula:Floquet}
\end{equation}

In the case $K = 2$, with $\theta := \theta_1$, the Perron-Frobenius theorem applies to
\[
 M(\theta) := e^{(1-\theta) T \cdot D F^2(0)} e^{\theta T \cdot D F^1 (0)},
\]
which is positive since $D F^k (0)$ are (irreducible) Metzler matrix by \eqref{monotone} (and \eqref{irre}).
Therefore there exists unique vectors $V (\theta), V_* (\theta) \gg 0$ with $\lVert V(\theta) \rVert = 1$ and $\langle V(\theta), V_* (\theta) \rangle = 1$, and a unique positive number $\rho (\theta)$ such that
\begin{equation}
 M(\theta) V(\theta) = \rho (\theta) V(\theta), \quad M(\theta)^* V_* (\theta) = \rho(\theta) V_* (\theta).
 \label{eq:M}
\end{equation}
In this setting, assume without loss of generality that $\mu(DF^2(0)) \geq \mu(DF^1(0))$, and denote $S := DF^1 (0) - DF^2 (0)$.
We consider two specific cases:
 \begin{enumerate}
  \item[(A)] $DF^1(0)$ and $DF^2(0)$ have the same principal right or left eigenvector;
  \item[(B)] for all $\theta \in [0, 1]$, one of the following holds:
  \begin{enumerate}
    \item[(B-1)] $\exists P \in GL_N (\R)$, $P S < 0$ and $(P^{-1})^* V_* (\theta) > 0$;
    \item[(B-2)] $\exists P \in GL_N (\R)$, $S P < 0$ and $P^{-1} V (\theta) > 0$;
    \item[(B-3)] $\exists P, Q \in \mathcal{M}_N(\R)$, $S < P^* Q$ and $P V_* (\theta) = - Q V (\theta)$.
  \end{enumerate}
 \end{enumerate}

\begin{theorem}
 Let $F$ of the form~\eqref{def:Fpiecewiseconstant} with $K = 2$ satisfy the assumptions of Theorem \ref{thm:extension}. Assume that the forward orbits of \eqref{F0} are bounded. Then under $(A)$ or $(B)$, \eqref{def:threshold} holds.  
 
 \label{thm:anydimension}
 
\end{theorem}

\begin{remark}
In addition, condition $(B-1)$ (resp. $(B-2)$) is equivalent to 
 \[
  S^* V_* (\theta) < 0 \text{ (resp. } S V (\theta) < 0 \text{)},
 \]
and if condition $(A)$ holds then $V(\theta) \equiv V$ or $V_* (\theta) \equiv V_*$, where $V$ (resp. $V_*$) is the right (resp. left) principal eigenvector of $DF^i(0)$, $i \in \{1, 2\}$.
 
\end{remark}

\begin{proof}
 We apply Theorem \ref{thm:extension} and check that the value of $\lambda$ (determining if case $(i)$ or $(ii)$ occurs) is a decreasing function of $\theta$ under assumptions $(A)$ or $(B)$. The forward-boundedness of orbits rules out the case $x \to +\infty$, thus leading to the result. More details in Section~\ref{subs:proofanydim}.
\end{proof}
\begin{remark}
 In the case $DF^2(0) > DF^1(0)$, we note that conditions $(B-1)$ and $(B-2)$ are obviously satisfied with $P = I$ (identity matrix), and condition $(B-3)$ is obviously satisfied with $P = Q = 0$.
\end{remark}
\begin{remark}
As will be seen below, in practical situations it is sometimes easier to check condition $(B-1)$ rather than computing $S^* V_*(\theta)$.
\end{remark}

\subsection{Application to a two-dimensional model of insect population dynamics}
We can now specify Theorem \ref{thm:anydimension} to the two-dimensional ($N=2$) case of \eqref{S0}.
First we describe the general properties of this system
\begin{proposition}
 For system \eqref{S0} written as $\dot{X} = G(\pi(t), X) =: F(t, X)$, where $\pi$ is defined by \eqref{def:pi}, assume that $\pi(t) \gg 0$, there exists $c, C \in \R_+^*$ such that $\pi_i (t) \geq c$ for $i \in \{4, 5\}$ and $\pi (t) \leq C \mathds{1}$.  Then, it is positive, forward-bounded, cooperative and concave.
 \label{prop:sysgen}
\end{proposition}
Then, we give the dynamics of the non-seasonal (=autonomous) system  \eqref{S0} with $\pi(t) \equiv \pi = (b, h, d_J, c_J, d_A)$. We define the basic offspring number:
    \begin{equation}
      \mathcal{R}_0 = \mathcal{R}(\pi) := \frac{b h}{d_A (h+d_J)}.
      \label{eq:R0}
    \end{equation}
\begin{proposition}
  If $\mathcal{R}_0 \leq 1$, then \eqref{S0} has no positive steady state and the trivial equilibrium is a global attractor.
    If $\mathcal{R}_0 > 1$ then \eqref{S0} has exactly one positive steady state $S_1^*=  (\mathcal{R}_0 - 1) \big(\frac{h+d_J}{c_J},\frac{h (h+d_J)}{c_J d_A} \big)$, which is a global attractor in $\R_+^2 \backslash \{ 0 \}$.
 \label{prop:autonomous}
\end{proposition}
The proofs of Proposition \ref{prop:autonomous} and Proposition \ref{prop:sysgen} are to be found in Section \ref{sec:autonomous}.

We finally state the sharp seasonal threshold property for \eqref{dyn:2seasons}:
\begin{theorem}
For \eqref{dyn:2seasons} under assumption \eqref{hyp:parameters}, if $\mathcal{R}_0 (\pi^U) < 1 < \mathcal{R}_0 (\pi^F)$ and $b^U + d_J^U > d_A^ U$ (where $\pi^U = (b^U, h^U, d_J^U, c_J^U, d_A^U)$) then \eqref{def:threshold} holds with $\theta_* \in (0, 1)$.
 \label{thm:main}
\end{theorem}
\begin{proof}
 We check assumption $(B-1)$  with
 \[
  P = \begin{pmatrix}
       1 & 1 \\
       0 & 1
      \end{pmatrix}, \quad
  (P^{-1})^* = \begin{pmatrix}
                1 & 0 \\
                -1 & 1
               \end{pmatrix}.
 \]
 More details in Section \ref{subs:proofN2}.
\end{proof}

\begin{remark}
    If instead of \eqref{hyp:parameters} we assume the stronger condition
    \begin{equation}
    \begin{pmatrix}
        -(h^F + d_J^F) + h^U + d_J^U & b^F - b^U \\
        h^F - h^U & -d_A^F + d_A^U
        \end{pmatrix}
        >0,
        \label{hyp:alternative}
    \end{equation}
    then assumption $(B-1)$ (or $(B-2)$) of Theorem \ref{thm:anydimension} applies with $P = I$ and no further computations are needed.
    
    We emphasize that \eqref{hyp:parameters} is more biologically relevant than \eqref{hyp:alternative}. The latter requires that the increase of the hatching rate between favorable and unfavorable season does more than compensate the decrease of juvenile death rate, which is highly debatable. This justifies the technical computations of Section \ref{subs:proofN2}.
    
    Note that in any case, no assumptions are made on $c_J^U$ and $c_J^F$, since the behavior is only determined by the linearization at $0$.
\end{remark}

\section{Proofs}
\label{sec:proofs}

\subsection{Proof of Theorem \ref{thm:anydimension}} 
\label{subs:proofanydim}

When there are only two dynamics within a period, that is when $K = 2$, we notice that the alternative $(i)-(ii)$ from Theorem \ref{thm:extension} is uniquely determined by the sign of the real function:
\[
 \theta \quad \mapsto \quad \rho \big( e^{(1-\theta) T \cdot D F^2(0)} e^{\theta T \cdot D F^1 (0)} \big) - 1.
\]

We notice that
\begin{lemma}
 The function $\rho : [0, 1] \to \R$ is $\mathcal{C}^1$ and satisfies
 \begin{equation}
 \rho'(\theta) = T \rho(\theta) \langle (D F^1 (0) - D F^2 (0)) V (\theta), V_* (\theta) \rangle.
 \label{eq:rhoprime}
\end{equation}
\label{lem:rhoprime}
\end{lemma}
\begin{proof}
 By Perron-Frobenius theorem, $\rho(\theta)$ is the maximal root of the characteristic polynomial of $M(\theta)$, whose entries are analytic functions of $\theta$. In particular, it is $\mathcal{C}^1$.
 
 The principal eigenvector of norm $1$ of $M(\theta)$, that is $V(\theta)$, depends smoothly of $\theta$, as can be seen by uniqueness for all $\theta$.
 Then, $V_*(\theta)$ also depends smoothly of $\theta$ since the same argument applies to $M^* (\theta)$ and $V_* (\theta)$ is equal to the principal eigenvector $Y_* (\theta)$ of $M^* (\theta)$ divided by $\langle V(\theta), Y_* (\theta) \rangle > 0$, which is a smooth function of $\theta$.
 
 Let us write $M_i := DF^i (0)$ for $i \in \{1, 2\}$. We differentiate the identity $\rho(\theta) = \langle M(\theta) V(\theta), V_*(\theta) \rangle$ to obtain
 \begin{align*}
  \rho'(\theta) &= \langle M(\theta) V'(\theta), V_* (\theta) \rangle + \langle M'(\theta) V(\theta), V_* (\theta) \rangle + \langle M(\theta) V(\theta), V'_* (\theta) \rangle,
  \\
  &= \rho (\theta) \Big( \langle V'(\theta), V_*(\theta) \rangle + T \big( \langle V(\theta), M_1^* V_* (\theta) \rangle - \langle M_2 V(\theta), V_* (\theta) \rangle \big) +  \langle V(\theta), V'_* (\theta) \rangle \Big),
  \\
  &= T \rho(\theta) \langle (M_1 - M_2) V(\theta), V_* (\theta) \rangle,
 \end{align*}
 since $M'(\theta) = T e^{(1-\theta )T M_2} \big( M_1 - M_2 \big) e^{\theta T M_1}$ and $\langle V(\theta), V_*(\theta) \rangle \equiv 1$.
\end{proof}

Applying Theorem \ref{thm:extension} with the assumption that the forward orbits are bounded, we are left with either global asymptotic stability of $0$ is $\lambda \leq 1$, or the global stability of the unique positive periodic solution, if $\lambda > 1$.
Using formula \eqref{formula:Floquet}, we obtain \eqref{def:threshold} with $\rho(\theta_*) = 1$ (or $\theta_* = 0$ if $\rho(0) >1 $, and $\theta_* = 1$ if $\rho(1) \leq 1$) if $\rho$ is a decreasing function of $\theta$.

It remains to prove that any of the conditions $(A)$ or $(B)$ implies that $\rho$ is decreasing.
Under assumption $(B-1)$, with $S= DF^1 (0) - DF^2 (0)$ we get by Lemma \ref{lem:rhoprime}
\[
\frac{\rho'(\theta)}{T \rho(\theta)} = \langle S V(\theta), V_*(\theta) \rangle = \langle PS V(\theta), (P^{-1})^* V_*(\theta) \rangle < 0,
\]
since $PS < 0$, $V(\theta) \gg 0$ and $(P^{-1})^* V_*(\theta) > 0$ by assumption.
Note that this condition is equivalent to $S^* V_*(\theta) < 0$. Reasoning by density of $GL_N(\R)$ in $\mathcal{M}_N (\R)$, we assume that $S$ is invertible and check that if $S^* V_* < 0$ then $P = - S^{-1}$ satisfies the assumption, and conversely if $P S = Q < 0$, upon writing $(P^{-1})^* = (Q^{-1})^* S^* $ we get $(Q^{-1})^* S^* V_* > 0$, and by multiplication by $Q^* < 0$ this implies $S^* V_* < 0$.
The argument is symmetrical for assumption $(B-2)$ and is omitted here.

Under assumption $(B-3)$ we get by Lemma \ref{lem:rhoprime}
\[
 \frac{\rho'(\theta)}{T \rho(\theta)} = \langle S V(\theta), V_*(\theta) \rangle 
 < \langle P_* (\theta) Q(\theta) V(\theta), V_* (\theta) \rangle = - \lVert Q (\theta) V(\theta) \rVert^2 \leq 0,
\]
since $V(\theta), V_* (\theta) \gg 0$ (for the inequality), and $P V_* = - Q V$ (for the equality).

Finally, under assumption $(A)$ we get that $V(\theta) \equiv V$ and $V_* (\theta) \equiv V_*$ where $V$ (resp. $V_*$) is the principal eigenvector (resp. left principal eigenvector) of $D F^1 (0)$ (which is the same as the one of $D F^2 (0)$).
In this case,
\[
 \frac{\rho'(\theta)}{T \rho(\theta)} = \langle S V, V_* \rangle = \mu(DF^1 (0)) - \mu(DF^2 (0)),
\]
whence the result.

\subsection{Proofs of Proposition \ref{prop:sysgen} and Proposition \ref{prop:autonomous}}
\label{sec:autonomous}
Recall that by definition,
\[
 \forall X \in \R^2, \quad F(t, X) =  G(\pi(t), X) := \begin{pmatrix}
                                         \pi_1 X_2 - (\pi_2 + \pi_3 + \pi_4 X_1) X_1
                                         \\
                                         \pi_2 X_1 - \pi_5 X_2
                                        \end{pmatrix}.
\]

We first proceed to the proof of Proposition \ref{prop:sysgen}. If $X_i = 0$ for some $i \in \{1, 2\}$, then since $\pi(t) \geq 0$, $F_i (t, X) \geq 0$. Therefore the system is positive.

We recall the notation $\pi = (b, h, d_J, c_J, d_A)$. We have:
\begin{equation*}
  D_X F= \begin{pmatrix} -h-d_J- 2 c_J J & b \\  h &  -d_A \end{pmatrix}.
\end{equation*}
Thus, $D_X F$ is a Metzler matrix, so \eqref{S0} is monotone cooperative.

To check the concavity property, let $X \gg Y$. We simply compute
 \begin{equation*}
 D_X F(t, X)-D_X F(t, Y)= \begin{pmatrix} 2 c_J (Y_1-X_1) & 0 \\  0 &  0 \end{pmatrix} > 0.
  \end{equation*}

Then, we proceed to the proof of Proposition \ref{prop:autonomous}.
Calculating the equations of nullclines
\begin{equation*}
\begin{array}{l}
b A-h J-d_J J-c_J J^2=0,\\
 h J-d_A A=0,\\
 \end{array}
\end{equation*}
immediately yields all steady states as:
 \begin{equation*}
 S_0^* =(0,0),  \quad
 S_1^* = (\frac{bh}{d_J}-h-d_J ) \big(\frac{1}{c_J},\frac{h}{c_J d_A} \big).
 \end{equation*}
 Then, the sign of both components of $S_1^*$ is equal to the sign of $\mathcal{R}_0-1$, whence the result.

The stability and local behavior of solutions is detailed in
\begin{proposition}
If $\mathcal{R}_0 \leq 1$ the unique equilibrium point $S_0^*=(0,0)$ is either a stable node (when $\mathcal{R}_0 < 1$) or a singular point of superior order and of attracting type (when $\mathcal{R}_0 = 1$), in which case all the orbits in the neighborhood of the $S_0^*$ tend to $S_0^*$ along direction $\theta_1 :=\arctan\frac{h+d_J}{b}$.

If $\mathcal{R}_0 > 1$, the equilibrium point $S_0^*=(0,0)$ is of saddle type, and the direction of unstable manifold  is $\frac{h+d_J-d_A+\sqrt {(h+d_J-d_A)^2+4bh}}{2b}$. The equilibrium point $S_1^*$ is a stable node.

\label{prop:stab}
\end{proposition}
\begin{proof}
We divide the proof into three parts, depending on the sign of $\mathcal{R}_0 - 1$.
\paragraph{When $\mathcal{R}_0 = 1$.} Then \eqref{S0} becomes
\begin{equation}\label{S2}
\begin{array}{l}
\displaystyle\frac{dJ}{dt}=-\frac{b h }{d_A}J+b A-c_J J^2,\\
\displaystyle\frac{dA}{dt}= h A-d_A A.
\end{array}
\end{equation}
The determinant  of its Jacobian matrix is 
 \begin{equation*}
\begin{vmatrix} 
 -\frac{bh}{d_A}& b\\
   h & -d_A
 \end{vmatrix}=0.
 \end{equation*}
Hence, the equilibrium point $S_0^*$ of system \eqref{S2}  is an isolated critical point of higher order.
 
Obviously, system \eqref{S2} is analytic in a neighborhood of the origin. By Theorem 3.10 on page 79 of  \cite{ZZT}, any orbit of  \eqref{S2} tending to the origin must tend to it spirally or along a fixed direction, which depends on the characteristic equation of system  \eqref{S2}.
First of all, we introduce the polar coordinates  $J=r \cos \delta $, $A=r \sin \delta $,  where  $\delta \in [0,\frac{\pi}{2} ]$, $r  \in \mathbb{R}_+$  and we get the relation
\begin{equation*}
\begin{cases}
\dot r=r^{-1}(J \dot J +A \dot A)= r^m[R(\delta)+o(1)], \\
\dot {\delta} =r^{-2}(J \dot A - A \dot J)=r^{m-1}[G(\delta)+o(1)].
\end{cases}
\end{equation*}
This yields
\begin{equation*}
\begin{cases}
\dot r=r(-\frac{bh}{d_A}\cos^2\delta+b\cos\delta\sin\delta+h\cos\delta\sin\delta-d_A\sin^2\delta- c_J r \cos^3\delta),\\
 \dot {\delta} =h\cos^2\delta-d_A\cos\delta\sin\delta+(h+d_J)\cos\delta\sin\delta-b\sin^2\delta+ c_J r \cos^2\delta\sin\delta.
 \end{cases}
\end{equation*}
Then the characteristic equation of system \eqref{S2}  takes the form
\begin{equation}\label{S3}
G(\delta)=h\cos^2\delta-d_A\cos\delta\sin\delta+(h+d_J)\cos\delta\sin\delta-b\sin^2\delta=0,
\end{equation}
and we have 
\begin{equation*}
R(\delta)=-\frac{bh}{d_A}\cos^2\delta+b\cos\delta\sin\delta+h\cos\delta\sin\delta-d_A\sin^2\delta.
\end{equation*}
After equation \eqref{S3},  we get 
\begin{equation}
(\frac{h+d_J}{b}\cos\delta-\sin\delta)(d_A\cos\delta+b\sin\delta)=0.
\end{equation}
Thus
\begin{equation*}
\begin{cases}
\tan\ \delta_1=\frac{h+d_J}{b}, \\
\tan\delta_2=-\frac{d_A}{b}.
\end{cases}
\end{equation*}

Clearly, $G(\delta) = 0$ has two real roots which we denote by $\delta_1$ and  $\delta_2$. By the results in section  2 of  \cite{ZZT}, we know that neither the case no orbit of system \eqref{S2} can tend to the critical point $S_0^*$ spirally nor the singular case (if $G(\delta)\equiv 0 $).

 The orbits of the system tend to the origin along a characteristic direction $\delta_i$, given by solutions of the equation \eqref{S3}.
 Since the system is positive we need to consider $\delta \in [0,\frac{\pi}{2} ]$,  so $\delta_1=\arctan\frac{h+d_J}{b}$ is in first orthant and the orbits of the system approach the origin along the direction $\delta= \delta_J$.

 \paragraph{When $\mathcal{R}_0 > 1$.} We now write the Jacobian matrix  $\Jac$ of the system
 \begin{equation*}
  \Jac := \begin{pmatrix} -h-d_J-2 c_J E& b \\  h &  -d_A \end{pmatrix},
\end{equation*}
and consider $\Jac_0$ and $\Jac_1$ are the Jacobian matrices respectively at equilibrium point $S_0^*$ and $S_1^*$. At ${S_0}^*$,
 \begin{equation*}
  \Jac_0= \begin{pmatrix} -h-d_J& b \\  h &  -d_J \end{pmatrix},
\end{equation*}
whose eigenvalues read
\begin{equation*}
\begin{array}{l}
\lambda_1 =\frac{-(h+d_J+d_A)+\sqrt \Delta}{2},\\
\lambda_2 =\frac{-(h+d_J+d_A)-\sqrt \Delta}{2},\\
\end{array}
\end{equation*}
where $\Delta:=(h+d_J+d_A)^2-4[(h+d_J)d_A-hb]>0$ (since  $(h+d_J)d_A-hb<0$).
Then 
\begin{equation*}
\begin{array}{l}
\lambda_1+\lambda_2 =-(h+d_J+d_A) < 0,\\
\lambda_1\lambda_2 = (h+d_J)d_A - hb < 0,\\
\end{array}
\end{equation*}
so that one eigenvalue is positive and the another one is negative: $S_0^*$ is a saddle point.

To find the direction of the stable manifold or unstable manifold at $S_0^*$, we write 
\begin{equation*}
\frac{\dot A}{\dot J}= \frac {dA}{dt}=\frac{h J-d_A A}{-hJ-d_J J+bA-c_J J^2}=\frac{h-\frac{A}{J}}{-h-d_J+b\frac{A}{J}-c_J J}.
\end{equation*}
Consider $(J,A)$ tending to $S_0^*$ and let  $k :=\frac{A}{J}$.
Then $k$ is a solution to
\begin{equation*}
k=\frac{h-d_A k}{-h-d_J+bk},
\end{equation*}
which leads to two solutions $(k_1, k_2) \in \R_+^* \times \R_-^*$ given by
\begin{equation*}
\frac{h+d_J-d_A \pm \sqrt {(h+d_J-d_A)^2+4bh}}{2b}.
\end{equation*}
Hence, the boundary lines are $A=k_1 J$ and $A=k_2 J$ and by unstable manifold theorem we know that $k_1$ is the direction of unstable manifold at $(0,0)$ 

Then, at equilibrium point $S_1^*$,
 \begin{equation*}
  \Jac_1= \begin{pmatrix} h+d_J-\frac{2bh}{d_A}& b \\  h &  -d_A \end{pmatrix},
\end{equation*}
whose eigenvalues $\lambda_1, \lambda_2$ are real and satisfy
\begin{equation*}
\begin{array}{l}
\lambda_1+\lambda_2 = h+d_J-\frac{2bh}{d_A}-d_A< 0,\\
\lambda_1\lambda_2 = -d_A(h+d_J) +bh>0.
\end{array}
\end{equation*}
This implies that the two eigenvalues are real and negative, hence $S_1^*$ is a stable node.

\paragraph{Finally, if $\mathcal{R}_0 < 1$.} Then at equilibrium point $S_0^*$
 \begin{equation*}
  \Jac_0= \begin{pmatrix} -h-d_J& b \\  h &  -d_A \end{pmatrix} .
\end{equation*}
Because $(h+d_J)d_A-hb>0$, the eigenvalues are such that
\begin{equation*}
\begin{array}{l}
\lambda_1+\lambda_2 =-(h+d_J+d_A) < 0,\\
\lambda_1\lambda_2 = (h+d_J)d_A - hb >0,\\
\end{array}
\end{equation*}
with also the discriminant $(-h-d_J+d_A)^2+4bh>0$, hence they are both negative and the equilibrium point $S_0^*$ is a stable node.
\end{proof}

\begin{remark}
In particular when $h=0$ (no hatching), and the trivial equilibrium point $S_0^*$ is a stable node.
\end{remark}

We now prove that all the orbits of \eqref{S0} are forward bounded.
\begin{lemma}
Let
\[
    \tau^* := \sup_{t \geq 0} \frac{h(t)}{d_A(t)}, \quad J^* := \sup_{t \geq 0} \frac{b(t) \tau^* - h(t) - d_J(t)}{c_J(t)}.
\]
Under the assumptions of Proposition \ref{prop:sysgen}, $\tau^*$ and $J^*$ are finite. For all $X_0 \in \R_+^2$ and all real number $L \geq \max(0,J^*)$ such that $X_0 \in \Omega_L := [0,L] \times [0, \tau^* L]$, the solution $X(t)$ of \eqref{S0} with initial data $X_0$ belongs to $\Omega_M$.
\end {lemma}

\begin{proof}
 Under the assumptions of Proposition \ref{prop:sysgen}, $c_J \geq c > 0$ and $d_A \geq c$ while all parameters are smaller than $C > 0$, hence $J^*$ and $\rho^*$ are finite.
 
 For $L > 0$ we define the area rectangle $\Omega_L$ surrounded by four line segments $\ell_i$ with outward normal vector $\nu_i$:
 \begin{equation*}
 \begin{array}{l}
 \ell_1=\{(J,A) |  J=0 ,0\leq  A\leq \tau^* L) \}, \quad \nu_1 = (-1, 0),
  \\[2pt]
\ell_2=\{(J,A) | J=L, 0\leq  A\leq \tau^* L) \}, \quad \nu_2 = (1, 0),
 \\[2pt]
 \ell_3=\{(J,A) | 0\leq J \leq L , A=0\}, \quad \nu_3 = (0, -1)
  \\[2pt]
 \ell_4=\{(J,A) | 0\leq J \leq L ,  A=\tau^* L\}, \quad \nu_4 = (0, 1).
\end{array}
\end{equation*}
To prove that $\Omega_L$ is positively invariant, since the system is positive, we only need to show that the scalar products of $\frac{dX}{dt}$ and $\nu_i$ on $\ell_i$ for $i \in \{2, 4\}$ are non-positive:
\begin{equation*}
\begin{array}{l}
    \nu_4 \cdot G(\pi, X) = h J - d_A \tau^* L \leq 0 \text{ since } J \leq L \text{ and } d_A \tau^* \geq h,
      \\
     \nu_2 \cdot G(\pi, X) = b A - h L- d_J L- c_J L^2.
\end{array}
\end{equation*}
Since $A < \tau^* L$, $\nu_2 \cdot G(\pi, X) \leq 0$ on $\ell_2$ as soon as $b \tau^* - h - d_J -c_J L \leq 0$, that is
\[
    L \geq \frac{b \tau^* - h - d_J}{c_J}.
\]
Upon taking $L \geq J^*$ this inequality is satisfied. For $L$ large enough such that $X_0 \in \Omega_L$, we have proved that for all $t > 0$, the solution $X(t)$ of \eqref{S0} belongs to $\Omega_L$.
\end{proof}

The Dulac (divergence) criterion ensures that the system has no limit cycle, since:
\[
\mathrm{div}(F)= -(h + d_J + c_J J + d_A) < 0.
\]
This concludes the proof.

\subsection{Proof of Theorem \ref{thm:main}}
\label{subs:proofN2}

Theorem \ref{thm:main} is a consequence of Theorem \ref{thm:anydimension}, condition $(B-1)$. To check this condition, we apply the following result (specific to the dimension $N = 2$) to the positive matrix $M(\theta)$:
\begin{lemma}
Let $S \in \mathcal{M}_2 (\R)$ be a positive matrix, and assume vector $W = (w_1,w_2) \gg 0$ satisfies $S^* W = \mu W$ for some $\mu > 0$ ({\it i.e.} $W$ is the principal eigenvector of $S^*$). Then, $w_2 > w_1$ if and only if
\begin{equation}
    s_{11} + s_{21} < s_{12} + s_{22},
    \label{hyp:Astar}
\end{equation}
Where $s_{11}$,  $s_{21}$,  $s_{12}$ and $s_{22}$ are the elements of matrix $S$.
\label{lem:Astar}
\end{lemma}
\begin{proof}
We write $S W=\mu W$ as
\begin{equation*}
\begin{cases}
s_{11}w_1+s_{21} w_2=\mu w_1,\\
s_{12}w_1+s_{22} w_2=\mu w_2,\\
\end{cases}
\iff \begin{cases}
s_{11}+s_{21}\frac{w_2}{w_1}=\mu, \\
s_{12} \frac{w_1}{w_2}+s_{22}=\mu.  \\
\end{cases}
\end{equation*}

If $0 < w_1 < w_2$, since $S \gg 0$ we deduce that $s_{11}+s_{21} <  \rho < s_{12}+s_{22}$.

Conversely, if $ s_{11}+s_{21} <  s_{12}+s_{22}$, subtracting the previous equalities we obtain
\[
\mu (1 - \frac{w_2}{w_1}) = s_{11} - s_{12} + \frac{w_2}{w_1}(s_{21}-s_{22}) < (s_{22} - s_{21})(1 - \frac{w_2}{w_1}).
\]
By contradiction, we assume that $w_2 < w_1$. Then $\mu < s_{22} - s_{21}$. Injecting this inequality into the previous equality we obtain
\[
s_{12} + \frac{w_2}{w_1} s_{22} < (s_{22} - s_{21}) \frac{w_2}{w_1},
\]
whence $s_{12} < - \frac{w_2}{w_1} s_{21}$, which contradicts $S > 0$. Hence $w_2 > w_1$.
\end{proof}

Lemma \ref{lem:Astar} is satisfied by $M(\theta)$, so that condition $(B-1)$ holds with $P = \begin{pmatrix} 1 & 1 \\ 0 & 1 \end{pmatrix}$. Indeed, $(P^{-1})^* = \begin{pmatrix} 1 & 0 \\ -1 & 1 \end{pmatrix}$ and $(P^{-1})^* V_* > 0$ with $V_* \gg 0$ if and only if $[V_*]_2 > [V_*]_1$, hence by \eqref{hyp:parameters} we have $P \big( DF^2 (0) - DF^1 (0) \big) < 0$.

The remaining of the proof is devoted to checking that $M_{12} (\theta) + M_{22} (\theta) - M_{11} (\theta) - M_{21} (\theta) > 0$.
To this aim, we diagonalize
\[
DF^1(0) =
\begin{pmatrix}
-h^U - d_J^U & b^U \\
h^U & - d_A^U
\end{pmatrix} \text{ and } DF^2(0) = \begin{pmatrix}
- h^F - d_J^F & b^F \\
 h^F & - d_A^F
\end{pmatrix}
\]
by
\[
DF^1(0) = P_U \begin{pmatrix} \lambda_U^+ & 0 \\ 0 & \lambda_U^- \end{pmatrix} P_U^{-1},\quad DF^2(0) = P_F \begin{pmatrix} \lambda_F^+ & 0 \\ 0 & \lambda_F^- \end{pmatrix} P_F^{-1},
\]
where for $\sharp \in \{U, F\}$,
\[
P_{\sharp} = \begin{pmatrix} 1 & 1 \\ x_{\sharp}^+ & x_{\sharp}^- \end{pmatrix}, \quad P_{\sharp}^{-1} = \frac{1}{x_{\sharp}^- - x_{\sharp}^+}\begin{pmatrix} x_{\sharp}^- & -1 \\ -x_{\sharp}^+ & 1 \end{pmatrix}
\]
and
\begin{align*}
\lambda_{\sharp}^{\pm} &= -\frac{1}{2}(h^{\sharp} + d_J^{\sharp} + d_A^{\sharp}) \pm \frac{1}{2} \sqrt{(h^{\sharp}+ d_J^{\sharp} - d_A^{\sharp})^2 + 4 h^{\sharp} b^{\sharp}}, 
\\
x_{\sharp}^{\pm} &= \frac{\lambda_{\sharp}^{\pm}+h^{\sharp}+d_J^{\sharp}}{b^{\sharp}}, \\
&= \frac{1}{2 b^{\sharp}}(h^{\sharp} + d_J^{\sharp} - d_A^{\sharp}) \pm \frac{1}{2 b^{\sharp}} \sqrt{(h^{\sharp}+ d_J^{\sharp} - d_A^{\sharp})^2 + 4 h^{\sharp} b^{\sharp}}.
\end{align*}

The condition of Lemma \ref{lem:Astar} will follow from:
\begin{lemma}\label{LM1}
 For $\sharp \in \{U, F\}$, we have $x_{\sharp}^-<0<x_{\sharp}^+$ and   $1+x_{\sharp}^->0$.
% (3): $x_2^--x_1^+ <0 $;\\
%  (4): $ x_1^--x_2^->0$;\\
  %(5): $x_1^+-x_2^+ <0$;\\ 
 % (6): $x_2^+-x_1^->0$.
\end{lemma}
\begin{proof}
The first inequalities follow directly from the above expression of $x_{\sharp}^{\pm}$. Then, we compute $1+x_{\sharp}^-=   \frac{2 b^{\sharp} + h^{\sharp} + d_J^{\sharp} - d_A^{\sharp} -\sqrt{(h^{\sharp} + d_J^{\sharp} - d_A^{\sharp})^2 + 4 h^{\sharp} b^{\sharp}}}{2 b^{\sharp}}$.
 We have
 \begin{align*}
(2 b^{\sharp}+ h^{\sharp} + d_J^{\sharp} - d_A^{\sharp} )^2 &= 4 (b^{\sharp})^2 + 4 b^{\sharp} (h^{\sharp} + d_J^{\sharp} - d_A^{\sharp}) + (h^{\sharp}+ d_J^{\sharp} - d_A^{\sharp})^2 
\\
&> (h^{\sharp} + d_J^{\sharp} - d_A^{\sharp})^2 + 4 h^{\sharp} b^{\sharp}
 \end{align*}
since $b^{\sharp} + d_J^{\sharp} - d_A^{\sharp} > 0$ (explicit assumption in Proposition \ref{prop:autonomous} for $\sharp = U$, and from $\mathcal{R}(\pi^F) > 1$ for $\sharp = F$). It implies  $1+x_{\sharp}^->0$.
%we define $z_1:= \sqrt{(h_1+ d_1 - d_2)^2 + 4 h_1 b}$  and $z_2:= \sqrt{(h_2+ d_1 - d_2)^2 + 4 h_2 b}$ \\
%However, letting $z(h) = \sqrt{(h+d_1-d_2)^2 + 4 h b}$, \[
%z'(h) = \frac{h + d_1 - d_2 + 2 b}{\sqrt{(h+d_1-d_2)^2 + 4 h b}}.
%\]
%we find that $z(0) = \lvert d_1 - d_2 \rvert$ and $z' > 1$ as soon as $b + d_1 - d_2 > 0$. Indeed,

%Under this assumption, we have $(z_1 - z_2)^2 > (h_1 - h_2)^2$, i.e $z_2 - z_1 > h_2 - h_1$\\
%Then we can find the last four inequalities  .
\end{proof}

 Thanks to the above diagonalization, we can write $M = M(\theta) = (m_{ij})_{1 \leq i, j \leq 2}$ as 
\begin{equation*}
\begin{array}{l}
 m_{11}=(\beta^+ x_F^- - \beta^-x_F^+)(\gamma^+x_U^- -\gamma^-x_U^+)+(-\beta^+ +\beta^-)(x_U^+x_U^- \gamma^+-x_U^+x_U^- \gamma^-),
 \\[5pt]
 m_{12}=(\beta^+x_F^- - \beta^-x_F^+)( -\gamma^+ +\gamma^-)+(-\beta^+ +\beta^-)(- x_U^+ \gamma^+ +x_U^- \gamma^-),
 \\[5pt]
 m_{21}=( x_F^+x_F^- \beta^+ -x_F^+x_F^- \beta^-)(\gamma^+ x_U^- - \gamma^- x_U^+ )+(-x_F^+ \beta^+ +x_F^- \beta^- )( x_U^+ x_U^- \gamma^+ - x_U^+ x_U^- \gamma^-),
 \\[5pt]
 m_{22}= (x_F^+x_F^- \beta^+-x_F^+x_F^- \beta^-)( -\gamma^++\gamma^-)+( -x_F^+ \beta^++x_F^- \beta^- )(-x_U^+ \gamma^++x_U^- \gamma^- ),
 \end{array}
\end{equation*}
where
\begin{align*}
\begin{array}{l}
 \beta^+:= e^{\lambda_F^+ (1 - \theta) T} , \quad \beta^-:= e^{\lambda_F^- (1 - \theta) T},
 \\[5pt]
  \gamma^+:= e^{\lambda_U^+ \theta T}, \quad  \gamma^-:= e^{\lambda_U^- \theta T},\\
  \alpha: =\df{b^U b^F }{\sqrt{\big( (h^U + d_J^U - d_A^U)^2 + 4 h^U b^U \big) \big( (h^F +d_J^F - d_A^F)^2 + 4 h^F b^F \big)}}.\\
  \end{array}
\end{align*}

Proving $m_{11}+m_{21}< m_{12}+m_{22}$ therefore amounts to checking
\begin{multline}
\beta^+\gamma^+(x_F^--x_U^+)(1+x_F^+)(1+x_U^-)+\beta^+\gamma^-(x_U^--x_F^- )(1+x_F^+)(1+x_U^+)\\
+\beta^-\gamma^+(x_U^+-x_F^+)(1+x_U^-)(1+x_F^-)+\beta^-\gamma^-(x_F^+-x_U^-)(1+x_F^-)(1+x_U^+)<0.
\label{ineq:betagamma}
\end{multline}
We introduce $\Psi : \R_+^2 \to \R$ as
\begin{multline*}
\Psi(\beta,\gamma) : =\beta\gamma(x_F^--x_U^+)(1+x_F^+)(1+x_U^-)+\beta(x_U^--x_F^- )(1+x_F^+)(1+x_U^+)\\
+\gamma(x_U^+-x_F^+)(1+x_U^-)(1+x_F^-)+(x_F^+-x_U^-)(1+x_F^-)(1+x_U^+),
\end{multline*}
so that \eqref{ineq:betagamma} is equivalent to $\Psi(\frac{\beta^+}{\beta^-} , \frac{\gamma^+}{\gamma^-}) < 0$.
First, it is easily checked that $\Psi(1,1)=0$, $ \beta^+>\beta^-$ and $ \gamma^+>\gamma^-$.
Then, by Lemma \ref{LM1}, $x_F^- < 0 < x_U^+$ and $1 + x_{\sharp}^{\flat} > 0$ for $\sharp \in \{ U, F\}$ and $\flat \in \{+, - \}$. Hence for $\beta > 1$,  we have
\begin{align*}
\frac{\partial\Psi(\beta,\gamma)}{\partial \gamma}  & = \beta(x_F^--x_U^+)(1+x_F^+)(1+x_U^-)+(x_U^+-x_F^+)(1+x_U^-)(1+x_F^-)\\
& < (x_F^--x_U^+)(1+x_F^+)(1+x_U^-)+(x_U^+-x_F^+)(1+x_U^-)(1+x_F^-) \\
&= (x_F^--x_F^+)(1+x_U^-)(1+x_U^+).
\end{align*} 
Symmetrically, for $\gamma > 1$ we have
\begin{align*}
\frac{\partial \Psi(\beta,\gamma)}{\partial \beta}  & = \gamma(x_F^--x_U^+)(1+x_F^+)(1+x_U^-)+(x_U^--x_F^- )(1+x_F^+)(1+x_U^+)\\
& < (x_F^--x_U^+)(1+x_F^+)(1+x_U^-)+(x_U^--x_F^- )(1+x_F^+)(1+x_U^+)\\
& =  (x_U^--x_U^+)(1+x_F^-)(1+x_F^+).
\end{align*} 

Applying Lemma \ref{LM1} again, we deduce that if $\beta, \gamma > 1$ then
\begin{equation*}
\frac{\partial \Psi}{\partial \gamma}, \frac{\partial \Psi}{\partial \beta} <0.
\end{equation*}
In particular $\Psi(\frac{\beta^+}{\beta^-} , \frac{\gamma^+}{\gamma^-}) < 0$, and this concludes the proof.

\section{Discussion and extensions}
\label{sec:discussion}

\paragraph{Geometric viewpoint.} We denote by $\Upsilon \times \Upsilon_*$ the graph of $\upsilon := (V, V_*) : [0, 1] \to (\R_+^*)^{2N}$. Then we define $r(\theta) := \frac{\rho'(\theta)}{T \rho(\theta)} = \langle S V(\theta), V_* (\theta) \rangle$.
Denoting by $\psi_S : \R^N \times \R^N \to \R$ the bilinear form $(V,W) \mapsto \langle A V, W \rangle$,  we get $r = \psi_S \circ \upsilon$.
Let $X_S := \{ \psi_S < 0 \}$, it is an open and radial subset of $\R^{2N}$ (if $Y \in X_S$ and $\lambda > 0$, then $\lambda Y \in X_S$). $\rho (M)$ is decreasing if and only if $r$ is decreasing, which is equivalent to $\Upsilon \times \Upsilon_* \subset X_S$. Up to changing $S$ into $-S$, assumption \eqref{simplecondition} amounts to $\upsilon(0), \upsilon(1) \in X_S$.

The case $(A)$  implies that $\Upsilon \times \Upsilon_*$ is a singleton, in which case \eqref{simplecondition} simply rewrites $(\mu_2 - \mu_1)^2 > 0$.

\paragraph{Practical computations in higher dimension.} Theorem \ref{thm:anydimension} suggests $4$ different sufficient conditions on $DF^1(0)$ and $DF^2(0)$ to obtain \eqref{def:threshold}. Apart from the trivial situations when $DF^1(0) - DF^2(0)$ has a sign or when the two matrices share the same principal eigenvector, how applicable are these conditions when $N > 2$ If $DF^i(0)$ is diagonalizable for $i \in \{1,2\}$, which we write
\[
    DF^i (0) = P_i^{-1} \diag((\lambda^{(k)}_i)_{1 \leq k \leq N}) P_i,
\]
then we can compute
\[
    M_{i,j} (\theta) = \sum_{j', j'' = 1}^N P_1^{-1} (i, j') Q(j',j'') P_2(j'',j) e^{T \big(\theta \lambda_1^{(j')} + (1 - \theta) \lambda_2^{(j'')} \big)},
\]
where $Q(j',j'') = \sum_{k=1}^N P_1(j',k) P_2^{-1} (k, j'')$.
For any matrix $\Gamma = (\gamma(i,j))_{1 \leq i,j \leq N} \in GL_N (\R)$ such that $\Gamma M(\theta) > 0$, we obtain $\Gamma V(\theta) > 0$ (where $V(\theta)$ is the principal eigenvector of $M(\theta)$). Then, a sufficient condition for \eqref{def:threshold} is given by $(DF^2(0) - DF^1(0)) \Gamma^{-1} < 0$.
Symmetrically, if $M(\theta) \Gamma > 0$ then a sufficient condition is given by $\Gamma^{-1} (DF^2(0) - DF^1 (0)) < 0$.

In order to get better conditions than the obvious ones, we require that $\Gamma \not\geq 0$.
We note that
\[
    \big[ \Gamma M(\theta) \big]_{i,j} = \sum_{k,j',j'' = 1}^N \gamma(i,k) P_1^{-1} (k, j') P_2(j'',j) Q(j', j'') e^{T \big(\theta \lambda_1^{(j')} + (1 - \theta) \lambda_2^{(j'')} \big)}.
\]

\paragraph{Log-convexity of the spectral radius.} A celebrated result of Kingman \cite{K} asserts that if the entries of a nonnegative matrix are log convex functions of a variable then so is the spectral radius of the matrix. If this property applies to the positive matrix $M(\theta)$, $\theta \mapsto \rho(M(\theta))$ is log-convex. In this case, it is monotone (yielding \eqref{def:threshold}) provided that the derivatives at $0$ and $1$ have the same sign, that is:
\begin{equation}
    \big( \mu_2 - \langle DF^1(0) V^2, V^2_* \rangle \big) \big( \langle DF^2(0) V^1, V^1_* \rangle - \mu_1 \big) > 0,
    \label{simplecondition}
\end{equation}
where $\mu_i = \mu(DF^i (0))$, and $V^i$ (resp. $V^i_*$) is the principal eigenvector of $DF^i(0)$ (resp. of $DF^i(0)^*$) with $V^i, V_*^i \gg 0$ and $\langle V^i, V^i \rangle = 1 = \langle V^i, V_*^i \rangle$.

When $DF^i(0)$ are diagonalizable ($i \in \{1, 2\}$), the above formula shows that
\[
    M_{i,j} (\theta) = \sum_{n=1}^{N^2} \alpha_n(i,j) e^{\beta_n(i,j) \theta}
\]
for some $\alpha, \beta$. In cases when $M_{i,j}$ can be proved to be a log-convex function of $\theta$, \eqref{def:threshold} holds under assumption \eqref{simplecondition}.

\paragraph{Computation of the second-order derivative.} A more general condition for \eqref{def:threshold} than the monotonicity of $\rho$ would be that $\rho$ is either concave or convex (or log-concave, or log-convex). To formulate this condition we compute the second-order derivative of $\log(\rho)$ from~\eqref{eq:rhoprime} as
\[
    \frac{d}{d\theta} \big( \log(\rho(\theta)) \big) = r'(\theta) =  \underbrace{\langle S V'(\theta), V_* (\theta) \rangle }_{=: R_1}+ \underbrace{ \langle S V(\theta), V_* '(\theta) \rangle}_{=: R_2},
\]
where
\begin{equation}
S=DF^1(0)-DF^2 (0). \label{eq:A}
\end{equation}
Differentiating with respect to $\theta$ the eigenvector equations for $V(\theta)$ and $V_*(\theta)$ along with their normalizations $\langle V(\theta), V(\theta) \rangle = 1$ and $\langle V(\theta), V_*(\theta) \rangle = 1$ yields:
\begin{align*}
(M(\theta) -\rho(\theta) I) V' (\theta) &= (\rho'(\theta) I - M'(\theta)) V(\theta),
\\
(M^*(\theta) -\rho(\theta) I) V'_* (\theta) &= (\rho'(\theta) I - (M^*)'(\theta)) V_*(\theta),
\\
\langle V(\theta), V'(\theta) \rangle &= 0 = \langle V'(\theta), V_*(\theta) \rangle + \langle V(\theta), V'_*(\theta) \rangle.
\end{align*}
Dropping the argument $\theta$, we note that $V', V'_*$ are well-defined from these linear equations since $\Ima(M - \rho I) = (V_* \R)^{\perp}$ (and symmetrically $\Ima(M^* - \rho I) = (V \R)^{\perp}$) and the scalar product conditions give uniqueness. We introduce the notation $H := (V \R)^{\perp}$ (resp. $H_* := (V_* \R)^{\perp}$) for the hyperplane with normal vector $V$ (resp. $V_*$).
We also introduce the Perron projection operator $\Pi :=V_*V^* \in \mathcal{L} (\R^N)$, and its adjoint $\Pi^*=V {V_*}^*$.

In particular, $M- \rho I \in \mathcal{L}(H ,H_*)$ is an invertible linear application, whose inverse is denoted $M_r \in \mathcal{L}(H_*, H)$, and we have
\[
    V' = M_r \big( (\rho' I - M' ) V \big).
\]
Symmetrically, $M^* - \rho I \in \mathcal{L}(H)$ is invertible (since $V_* \not\in H$), its inverse is denoted $M_a^* \in \mathcal{L}(H)$ and
\[
    V'_* = M_a^* \big( (\rho' I - M'_*) V_* \big) - \langle V_*, V' \rangle V_*.
\]
Using the notation $M_i = DF^i (0)$ ($i \in \{1, 2\}$), from the definition $M(\theta) = e^{T (1 - \theta) M_2} e^{T \theta M_1}$ we also have:
 \begin{align}
 M' &=T(MM_1-M_2M), \label{eq:Mprime}
 \\
 (M^*)'&=T\big((M_1)^*M^*-M^*(M_2)^*\big). \label{eq:Mstarprime}
 \end{align}
In order to compute the two terms in $r'$,
we note two preliminary identities. First, using \eqref{eq:Mprime} and \eqref{eq:rhoprime} we get
\begin{equation}
\frac{1}{T} (\rho' I - M') V  = \rho (\Pi^* - I) S V - (M - \rho I) M_1 V,
\label{eq:rhoprimeMprime}
\end{equation}
where both terms in the right-hand side belong to $H_*$.
Symmetrically, using \eqref{eq:Mstarprime} and \eqref{eq:rhoprime} we get
\begin{equation}
\frac{1}{T} (\rho' I - (M^*)') V_* = (M^* - \rho I) M_2^* V_* + \rho (\Pi - I) S^* V_*,
\label{eq:rhoprimeMstarprime}
\end{equation}
where both terms in the right-hand side belong to $H$.

Then, using \eqref{eq:rhoprimeMprime}, $M_r \in \mathcal{L} (H_*, H)$ and $M_r \circ (M - \rho I) = I_{H}$ we can compute
\begin{align*}
R_1 & = \big\langle M_r \big( (\rho' I - M' ) V \big), S^*V_*  \big\rangle, \\
        &= T \rho \big\langle M_r (\Pi^* - I) S V, S^* V_* \big\rangle - T \langle M_1 V, S^* V_* \rangle.
\end{align*}

Symmetrically, using \eqref{eq:rhoprimeMstarprime}, $M_a^* \in \mathcal{L} (H)$ and $M_a^* \circ (M^* - \rho I) = I_H$ we obtain
\begin{align*}
 R_2  & = \langle S V,  M_a^* \big( (\rho' I - M'_*) V_* \big) - \langle V_*, V' \rangle V_*   \rangle,
        \\
        &= T \rho \big\langle S V, M_a^* (\Pi - I) S^* V_* \big\rangle + T \langle S V, M_2^* V_* \rangle - \langle S V, V_* \rangle \langle V_* , V' \rangle.
\end{align*}

Using \eqref{eq:rhoprimeMprime} with $M_r \in \mathcal{L} (H_*, H)$ and $(M-\rho I) \circ M_r = I_{H}$ we also get
\begin{align*}
 \langle V_*, V' \rangle & =  \langle V_*,  M_r \big( (\rho' I - M' ) V \big) \rangle,
 \\
 &= T \rho \big\langle V_*, M_r (\Pi^* - I) S V \big\rangle - T \langle V_* , M_1 V \rangle.
\end{align*}
Gathering $R_1$ and $R_2$ we obtain
\begin{multline*}
    \frac{r'}{T} = \overbrace{\big( \langle S V, V_* \rangle \big)^2 + \big\langle (M_2 S - S M_1) V, V_* \big\rangle}^{r_1} + \\ \underbrace{\rho \big\langle M_r (\Pi^* - I) S V, (S^* -\langle S V, V_* \rangle I) V_* \big\rangle}_{r_2} + \underbrace{\rho \big\langle M_a^*(\Pi-I) S^* V_*, S V \big\rangle}_{r_3}.
\end{multline*}

We notice that
\[
r_2 = \rho \big\langle M_r (\Pi^* - I) S V, (I - \Pi) S^* V_* \big\rangle = \rho \big\langle SV, (I-\Pi ) M_r^* (\Pi- I) S^* V_* \big\rangle
\]
and
\[
r_3 =   \rho \big\langle SV, M_a^* (\Pi- I) S^* V_* \big\rangle,
\]
so $r_2 = r_3$, since $(M^*-\rho I) \circ M_a^*=I_{H}$,  $(M^*-\rho I) \circ M_r^*=I_{H}$ and $ (M^*-\rho I) \circ  \Pi  M_r^*=0$

Finally $\rho'' =T^2\rho r^2+ T\rho r'$ whence
\begin{equation}
    \frac{\rho''}{T^2 \rho} = 2 \big( \langle S V, V_* \rangle \big)^2 + \big\langle (M_2 S - S M_1) V, V_* \big\rangle + 2 \rho \big\langle M_a^*(\Pi-I) S^* V_*, S V \big\rangle.
    \label{eq:rhosecond}
\end{equation}

In principle, the identity \eqref{eq:rhosecond} could be used to derive \eqref{def:threshold} under more general conditions on $M_1 = DF^1(0), M_2 = DF^2 (0)$ than those given in Theorem \ref{thm:anydimension}. However, we do not explore such conditions in the present article.

\paragraph{Time scaling.} Until now we have considered that the period $T > 0$ was fixed. Letting $T$ go to $0$ or $+\infty$ yields interesting limits. For an irreducible Metzler matrix $U$,
\[
    e^{-T \mu(U)} e^{T U} \xrightarrow[T \to +\infty]{} V V_*^*
\]
where $V$ is the principal eigenvector of $U$ and $V_*$ is the principal eigenvector of $U^*$, normalized by $V_*^* V = 1$. From this fact, we have
\[
    e^{-T (\theta\mu(DF^1(0))+(1-\theta) \mu(DF^2(0))}M(\theta) \xrightarrow[T \to +\infty]{} V(0)V_*(0)^* V(1) V_*(1)^*,
\]
from which we deduce that
\[
    \frac{1}{T}\log(\rho(\theta)) \sim_{T \to +\infty} \theta\mu(DF^1(0))+(1-\theta) \mu(DF^2(0)).
\]
In fact, we even get the next term in the asymptotic development:
\[
    \log(\rho(\theta)) - T \big( \theta\mu(DF^1(0))+(1-\theta) \mu(DF^2(0)) \big) - \log \big(V_*(0)^* V(1) V_*(1)^* V(0) \big) = o_{T \to \infty}(1).
\]
Therefore, for $T$ large enough, $\rho$ is close to be monotone, and even close to be equal to the exponential interpolation of $T \mu(DF^1(0))$ and $T \mu(DF^2(0))$.

Meanwhile, $\lim_{T \to 0} \rho(\theta) \equiv 1$.

\paragraph{Optimization problems.} For a general two-seasonal model defined by a monotone and concave map $G : \mathcal{P} \times \R^N \to \R^N$ and $\pi^U, \pi^F \in \mathcal{P}$, a natural question is the optimization of the spectral radius when the favorable and unfavorable seasons can be split throughout the year. Let $M_{\sharp} := T \cdot DG(\pi_{\sharp}, 0)$ (with $\sharp \in \{U, F\}$). For $K \in \Z_+$, we define:
\begin{align}
\overline{\rho}_{M_U,M_F} (\theta,K) &= \max_{(\sigma, \sigma') \in \varphi_K(\theta)} \rho(M_{M_U, M_F}(\sigma,\sigma')),
\label{def:maxirho}
\\[10pt]
\underline{\rho}_{M_U,M_F} (\theta,K) &= \min_{(\sigma, \sigma') \in \varphi_K(\theta)} \rho(M_{M_U, M_F}(\sigma,\sigma')),
\label{def:minirho}
\end{align}
where
\[
    \varphi_K (\theta) := \big\{ \big((\theta_k)_k, (\theta'_k)_k \big) \in [0, 1]^{2K}, \, \sum_{k=1}^K \theta_k = \theta, \, \sum_{k=1}^K \theta'_k = 1 - \theta \big\}
\]
is compact and for $(\sigma,\sigma')\in \varphi_K(\theta)$ and $M_1, M_2 \in \mathcal{M}_N (\R)$,
\[
    M_{M_1, M_2} (\sigma, \sigma') := e^{\theta'_K M_2} e^{\theta_K M_1} \cdots e^{\theta'_1 M_2} e^{\theta_1 M_1}.
\]
Note that by Gelfand's formula, 
\[
    \rho(M(\sigma,\sigma')) \leq \prod_{k} \rho(e^{\theta'_k M_2}) \rho(e^{\theta_k M_1}) = e^{\theta \mu_1 + (1 - \theta) \mu_2},
\]
where $\mu_i = \mu(M_i)$.

\begin{remark} In the specific case when $M_U$ and $M_F$ are irreducible Metzler matrices with the same principal eigenvector (that is, condition $(A)$) , $\rho(M(\sigma,\sigma'))$ does not depend on $(\sigma, \sigma') \in S_K (\theta)$ and does even not depend on $K \in \Z_+$: we have
\[
    \forall K \in \Z_+, \forall \theta \in [0, 1], \quad \overline{\rho}_{M_U,M_F} (\theta,K) = e^{\big( \theta \mu_U + (1 - \theta) \mu_F \big)} = \underline{\rho}_{M_U,M_F} (\theta,K),
\]
with $\mu_{\sharp} = \mu(M_{\sharp})$.

In this case, assuming $\mu_F > 0 > \mu_U$ we recover Theorem \ref{thm:main} with 
\[
    \theta_* = \frac{\mu_F}{\mu_F - \mu_U}.
\]
\end{remark}

\paragraph{Acknowledgements.} The authors wish to thank Dongmei Xiao and Jean-Pierre Fran\c{c}oise for useful discussions, and Benoit Perthame for valuable comments which helped to improve this manuscript. Part of this work was done while HJ was visiting Dongmei Xiao at SJTU, he thanks sincerely all the members of ODE\&DS group.

\appendix

\section{Proof of Theorem \ref{thm:extension}}
\label{app:proof}

We consider the following $T$-periodic piecewise-autonomous differential equation
\begin{equation}\label{AP1}
\frac{dx}{dt}=F(t,x),
\end{equation}
where for all $x \in \R^N$, $F(\cdot, x)$ is a piecewise-constant function. We assume that there is a family of functions $(F^k)_k : \R_+^N \to \R_+^N$ such that:
 \begin{align*}
F(t,x) &= F^{k} (x) \text{ if } \frac{t}{T} - \big\lfloor \frac{t}{T} \big\rfloor \in [\theta_{k-1}, \theta_k)
%\begin{cases}
%F^1(x),  \quad   nT \leq t < (n+ \theta_1)T \\
%\quad \vdots \\
%F^k(x),  \quad    (n+\theta_{k-1})T \leq  t <  (n+\theta_k)T\\
%\quad \vdots \\
%F^K(x),  \quad    (n+\theta_{N-1})T \leq  t < (n+1)T \\
% \end{cases}
 %\\
 %&=  \sum\limits_{k=1}^{K} \mathds{1}_{[(n+\theta_{k-1})T,   (n+\theta_k)T)}(t)F^k(x),
 \end{align*}
where $(\theta_i)_{0 \leq i \leq N} \in [0, 1]^{N+1}$ is a non-decreasing family such that $\theta_0 = 0$ and $\theta_N = 1$. For $x \in \R$, the notation $\lfloor x \rfloor$ stands for the largest integer $n \in \Z$ such that $n \leq x$.
 
We assume that for all $1 \leq k \leq K$,  $F^k$ : $\R_+^N \to  \R_+^N$ is continuously differentiable, monotone (that is, if $x \ll y$ then $F^k (x) \ll F^k(y)$), concave (that is, if $x \ll y$ then $DF^k(x) \gg DF^k (y)$) and satisfies $F^k (0) = 0$.

Following the lines of \cite{HLS1} and \cite{JJF}, to prove Theorem  \ref{thm:extension} we split into four assertions the various hypotheses of \cite[Theorem 2.1]{HLS1}, to check that they hold for the Poincare map for \eqref{AP1}. We begin with:

\begin{lemma}   \label{A1}
If $x(t)$ is a solution of \eqref{AP1} with $x(t_0) \geq 0$, then $x(t)$ can be extended to  $[t_0, +\infty]$ and $x(t) \geq 0 $ for $t \geq t_0$.
\end{lemma}
\begin{proof}
Let $t \geq 0$. For all $y \geq 0$, by concavity of all $F^k$ ($1 \leq k \leq K$), we have $D_x F (t, y) \leq D_x F(t, 0)$. Hence for all $t \geq 0$ and $x \geq 0$,
\begin{align*}
F(t,x)&= F(t, 0) + \big( \int_0^1 D_x F(t, sx) ds \big) x
\\
& \leq F(t,0)+ D_xF(t,0) x \text{ since } x \geq 0.
\end{align*}

Let $y$ be the solution to the affine differential equation $y'=  F(t,0)+ D_x F(t,0)y$, $y(t_0) = x(t_0)$. From Kamke's theorem, we deduce that $x(t) \leq y(t)$ on the maximal interval of existence $[t_0,w)$ of $x(t)$. Since $y(t)$ is defined for all $t \geq t_0$, it follows that $w =+\infty$. 

The standard positivity property \eqref{positive} implies $x(t) \geq 0$ for $t\geq t_0$.
\end{proof}

Then, as an immediate consequence of monotonicity and Kamke's theorem:
\begin{lemma}  \label{A2}
If $x(t)$ and $y(t)$ are solutions of \eqref{AP1} with $0 \leq y(t_0) \ll x(t_0)$, then $y(t) \ll x(t)$ for $t > t_0$.
\end{lemma}

For all $s \in \R$ and $x_0 \in \R^N$, we denote by $t \mapsto \phi(t; s, x_0)$ the solution of \eqref{AP1} which satisfies $x(s) = x_0$. In particular, $\phi(s;s, x) = x$.
For all $1 \leq k \leq K$, we also introduce $t \mapsto \phi^k(t; s,x_0)$ as the solution to 
\[
\frac{dx}{dt} = F^k(x), \quad x(s) = x_0.
\]
By regularity of $F^k$, each $\phi^k(\theta_k T, \theta_{k-1} T, \cdot)$ is a $C^1$ function.

With these notations it follows from Lemmas~\ref{A1} and \ref{A2} that the Poincare map
\begin{equation}
P(x) := \phi(T; 0, x) = \phi^K \big( \theta_K T;  \theta_{K-1}T,\phi^{K-1} \big( \cdots \phi^1(\theta_1 T; 0 , x) \big) \big), \quad  x\geq 0
\end{equation}
is well defined as a  $C^1$ map $P : \R_+^N \to  \R_+^N$ because  it is a composition of functions of class $C^1$. In order to apply \cite[Theorem 2.1]{HLS1}, we must verify that the differential $DP$ satisfies:
\begin{align}
\tag{$M_0$} &DP(0) \gg 0 \text{ and } DP (x) \geq 0 \text{ if } x \gg 0,
\label{DP:M0}
\\
\tag{$C_0$} & DP(y) < DP(x) \text{ if } 0 \ll x \ll y.
\label{DP:C0}
\end{align}
Introducing the notations, for $x \in \R^N$
\begin{align*}
&\widetilde{\phi}^k (x) := \phi^k \big( \theta_k T; \theta_{k-1} T, \widetilde{\phi}^{k-1} (x) \big) \in \R^N \text{ for }1 \leq k \leq K, \quad \widetilde{\phi}_0 (x) := x,
\\
& \widehat{\phi}^k (x) := \frac{\partial \phi^k}{\partial x} (\theta_k T; \theta_{k-1} T, x) \in \R^{N \times N},
\end{align*}
we can compute
\begin{equation}
DP(x)= \frac{\partial \phi}{\partial x }(T; 0, x)= \prod_{k= 1}^K \widehat{\phi}^k \circ \widetilde{\phi}^{k-1} (x).
\end{equation}

We write $\Phi(t, x) := \frac{\partial \phi}{\partial x }(t; 0, x)$, so that $DP = \Phi(T,\cdot)$. By construction, $\Phi(t, x)$ is the fundamental matrix for the variational equation
 \begin{equation} \label{C3}
X'=D_x F(t, \phi(t; 0, x) )X, \quad X(0)=I
\end{equation}
where $I$ is the $N \times N$ identity matrix. Lemma \ref{A3} below is a direct consequence of \eqref{monotone}
\begin{lemma}  \label{A3}
If $x \gg 0$, then  $\Phi(t,x) > 0$ for $t>0$. In addition, $\Phi(t,0) \gg 0$ for $t > 0$.
\end{lemma}
\begin{proof}
    Let $T > 0$ and $x \in \R^N$. Let $M = M_{T,x} \in (0, +\infty)$ such that $D_x F (t, \phi(t;0,x)) + M I \geq 0$ for all $t \in [0, T]$. As long as $\Phi(t,x) \geq 0$ on $[0, T]$ we have on this interval $ \frac{d}{dt} \Phi(t,x) \geq -M \Phi(t,x)$, hence $\Phi(t,x) \geq e^{-M t} I > 0$.
    
    Then, $\Phi(t,0)$ solves \eqref{dyn:linearized} with $\Phi(0, 0) = I$. Since $D_x F(t,0)$ is an irreducible (by \eqref{irre}) Metzler matrix, $\Phi(t,0) \gg 0$ for $t > 0$.
\end{proof}
Applying Lemma \ref{A3} with $t = T$ yields \eqref{DP:M0}.
It remains only to verify \eqref{DP:C0}, which is the object of the next lemma
 \begin{lemma}  \label{A4}
If $0 \ll x \ll y$, then $DP(x)>DP(y)$.
\end{lemma}
\begin{proof}
We write $Z(t, x) = D_x F(t, \phi(t; 0, x) )$ for short. If $0 \ll x \ll y$, from Lemma \ref{A2}, we have $\phi(t; 0, x) \ll \phi(t; 0, y)$ for all $t \geq 0$.  By \eqref{concave}, we deduce that $Z(t, x) > Z(t,y)$. Hence
\begin{align*}
\Phi '(t,x)&= Z(t, x) \Phi (t,x)\\
&\geq Z(t, y) \Phi (t,x),
\end{align*}
since $\Phi(t,x) \geq 0$ by Lemma \ref{A3}.
Therefore, it follows from Kamke's theorem that $\Phi (t,x) \geq \Phi (t,y)$. 

Then, we follow (\cite{AM}, lemma l) by letting $Y(t) = \Phi (t,x)- \Phi (t,y)$. $Y(t)$ satisfies
\begin{equation*}
Y'(t)=Z(t,x)Y(t)+[Z(t, x) -Z(t, y) ]\Phi (t,y), \quad Y(0)=0.
\end{equation*}
Using the fundamental matrix $\Phi$ we get
\begin{equation*}
Y(T)=\int_{0}^{T}   \Phi (T,x) {\Phi (s,x) }^{-1} [Z(s, x) -Z(s, y)]  \Phi (s,y)ds
\end{equation*}
Now, $ Z(t,s) \equiv  \Phi (t,x) {\Phi (s,x) }^{-1} > 0$ for $t>s$ since it is the fundamental matrix at $t= s$ of $z'= Z(t, x)z$ (exactly as in Lemma \ref{A3}). 
Since $\Phi(s, y) > 0$ for $0 < s \leq T$ and $Z(s, x) - Z(s, y) \gg 0$  for $0  \leq s \leq T$,  it follows that $Y(T) > 0$. This is the desired conclusion.
\end{proof}

We have verified all assumptions and can apply \cite[Theorem 2.1]{HLS1} and Theorem \ref{thm:extension} follows immediately  on noting that $\lambda = \rho(DP(0)) = \rho(\Phi(T, 0))$ is the characteristic multiplier of~\eqref{dyn:linearized} of maximum modulus.

\end{document}